\documentclass[final,a4paper,12pt]{amsart}
\usepackage{version} 
\usepackage[notcite,notref]{showkeys}
\usepackage{amssymb}
\usepackage{amsfonts}
\usepackage{cite}
\usepackage{color}

\newtheorem{theorem}{Theorem}[section]
\newtheorem{lemma}[theorem]{Lemma}
\newtheorem{proposition}[theorem]{Proposition}
\newtheorem{corollary}[theorem]{Corollary}

\theoremstyle{definition}
\newtheorem{definition}[theorem]{Definition}
\newtheorem{example}[theorem]{Example}

\theoremstyle{remark}
\newtheorem{remark}[theorem]{Remark}

\numberwithin{equation}{section}

\DeclareMathOperator{\Cov}{Cov}
\DeclareMathOperator{\Var}{Var}

\DeclareMathOperator{\por}{por}

\DeclareMathOperator{\lpor}{\underline{por}}
\DeclareMathOperator{\upor}{\overline{por}}
\DeclareMathOperator{\diam}{diam}
\DeclareMathOperator{\card}{card}

\def\1{{\rm1\!\!1}}
\def\e{\varepsilon}
\def\N{{\rm I\! N}}
\def\R{{\rm I\! R}}

\begin{document}

\title[Porosities of Mandelbrot percolation] 
{Porosities of Mandelbrot percolation}
\author[A. Berlinkov]{Artemi Berlinkov$^1$}
\address{Department of Mathematics, Bar-Ilan University, Ramat Gan, 5290002, Israel$^1$}
\email{berlinkov@technion.ac.il$^1$}
\thanks{AB was partially supported by the Department of Mathematics at 
University of Jyv\"askyl\"a; DFG-Graduirtenkolleg ``Approximation und 
algorithmische Verfahren'' at the University of Jena; University ITMO; ISF 
grant 396/15; Center for Absorption in Science,  Ministry of Immigrant 
Absorption, State of Israel}

\author[E. J\"arvenp\"a\"a]{Esa J\"arvenp\"a\"a$^2$}
\address{Department of Mathematical Sciences, PO Box 3000, 90014 University 
         of Oulu, Finland$^2$}
\email{esa.jarvenpaa@oulu.fi$^2$}
\thanks{EJ acknowledges the support of the Centre of Excellence in
Analysis and Dynamics Research funded by the Academy of Finland and 
thanks the ICERM semester program on ``Dimension and Dynamics''.}
\thanks{We thank Maarit J\"arvenp\"a\"a for interesting discussions and 
many useful comments.} 
\date{\today}

\dedicatory{}


\begin{abstract}
We study porosities in the Mandelbrot percolation process. We show that, almost 
surely at almost all points with respect to the natural measure, the mean 
porosities of the set and the natural measure exist and are equal to each other
for all parameter values outside of a countable exceptional set. As a 
corollary, we obtain that, almost surely at almost all points, the lower 
porosities of the set and the natural measure are equal to zero, whereas the 
upper porosities obtain their maximum values.
\end{abstract}

\maketitle


\section{Introduction}\label{intro}

The porosity of a set describes the sizes of holes in the set. The concept 
dates back to the 1920's when Denjoy introduced a notion which he called index 
(see \cite{De}). 
In today's terminology, this index is called the upper porosity (see 
Definition~\ref{defporoset}). The term porosity was introduced by Dol\v zenko 
in \cite{Do}. Intuitively, if the upper porosity of a set equals $\alpha$, 
then, in the set, there are
holes of relative size $\alpha$ at arbitrarily small distances.
On the other hand, the lower porosity (see 
Definition~\ref{defporoset}) guarantees the existence of holes of certain 
relative size at all sufficiently small distances.
The upper porosity turned out to be useful in order to describe properties 
of exceptional sets, for example, for measuring sizes of sets where certain
functions are non-differentiable. For more details about the upper porosity,
we refer to an article of Zaj\'\i\v cek \cite{Z}.
Mattila \cite{M} utilised the lower porosity to find upper bounds for Hausdorff
dimensions of set, and Salli \cite{S} verified the corresponding results for 
packing and box counting dimensions. 

It turns out that upper porosity cannot be used to estimate the dimension
of a set  (see \cite[Section 4.12]{M2}). An observation that there are sets 
which are not lower porous but  
nevertheless contain so many holes that their dimension is smaller than
the dimension of the ambient space, leads to the concept of mean
porosity of a set introduced by Koskela and Rohde \cite{KR} in order to study 
the boundary behaviour of conformal and quasiconformal mappings.
Mean porosity guarantees that certain percentage
of scales, that is, distances which are integer powers of some fixed number,
contain holes of fixed relative size. Koskela and Rohde showed that, if a 
subset of the $m$-dimensional Euclidean space is mean porous,
then its Hausdorff and packing dimensions are smaller than $m$. For a 
modification of their definition, see Definition \ref{defmean}.

The lower porosity of a measure was introduced by 
Eckmann and E. and M. J\"arvenp\"a\"a in \cite{EJJ}, the upper one by 
Mera and Mor\'an in \cite{MM} and the mean porosity by Beliaev and Smirnov in
\cite{BS}. The relations between porosities and dimensions of sets 
and measures have been investigated, for example, in
\cite{BS,BJJKRSS,JJ,JJ2,JJKRRS,JJKS,KRS,Shm2}. For further information on this 
subject, we refer to a survey by Shmerkin \cite{Shm}.
Porosity has also been used for studying the conical densities of measures 
(see \cite{KS,KS2}).
 
Note that sets with same dimension may have different 
porosities. In \cite{JJM}, E. and M. J\"arvenp\"a\"a and Mauldin and, in 
\cite{U}, Urba\'nski characterised deterministic iterated function systems 
whose attractors have positive porosity. Porosities of random recursive 
constructions were studied in \cite{JJM}. Particularly interesting random
constructions are those in which the copies of the seed set are glued 
together in such a way that there are no holes left. Thus, the corresponding 
deterministic system would be non-porous and the essential question is whether
the randomness in the construction makes the set or  measure porous.
A classical example is the Mandelbrot percolation process (also known as the 
fractal percolation) introduced by Mandelbrot
in 1974 in \cite{Man} (see Section \ref{perco}). In \cite{JJM}, it was shown 
that, almost surely, the points with minimum porosity as well as 
those with maximum porosity are dense in the 
limit set. However, the question about porosity of typical points and that of 
the natural measure remained open. Later, it turned out \cite{CORS}
that, for typical points, the lower porosity equals 0 and the upper 
one is equal to $\frac 12$ as conjectured in \cite{JJM}. Indeed, this is a 
corollary of the results of Chen et al. in \cite{CORS} dealing with estimates 
on the dimensions of sets of exceptional points regarding the porosity. 

In this paper, we prove that the mean porosities of the natural measure and of 
the limit set exist and are equal to each other almost surely at almost all
points with respect to the natural measure for all parameter values outside of
a countable set (see Theorem~\ref{equal}). We also show that mean porosities 
are continuous as a function of parameter outside this exceptional set (see
Theorem~\ref{meanporosityexists}). Unlike the upper and lower porosities, the 
mean porosities of the set and the natural measure at typical points 
are non-trivial. Indeed, we prove that almost surely 
the mean porosities of the set and the natural measure are positive and less 
than one for almost all points with respect to the natural measure for all 
non-trivial parameter values (see Corollaries~\ref{positive} and 
\ref{mainmeasure}). As an application of our results, we 
solve the conjecture of \cite{JJM} completely (and give a new proof for the 
part solved in \cite{CORS}) by showing that, almost surely for almost all points
with respect to the natural measure, the lower porosities of the
limit set and of the measure are equal to the minimum value of 0,
the upper porosity of the set attains its maximum value of
$\frac 12$ and the upper porosity of the measure also attains its
maximum value of 1 (see Corollary \ref{infsup}). 

The article is organised as follows. In Section~\ref{perco}, we explain some
basic facts about the Mandelbrot percolation and, in Section~\ref{porosities}, 
we define porosities and mean porosities and describe some
of their properties. Finally, in Section~\ref{results}, we prove
our results about mean porosities of the limit set and of the natural measure 
in the Mandelbrot percolation process.

\section{Mandelbrot percolation}\label{perco}

We begin by recalling some basic facts about Mandelbrot percolation in the 
$m$-dimensional Euclidean space $\R^m$, where 
$m\in\N=\{1,2,\dots\}$. Let $k\ge 2$ be an integer, 
$I=\{1,\dots,k^m\}$ and $I^*=\bigcup_{i=0}^\infty I^i$, where $I^0=\emptyset$. An 
element $\sigma\in I^i$ is called a word and its length is $|\sigma|=i$. For 
all $\sigma\in I^i$ and $\sigma'\in I^j$, we denote by
$\sigma\ast\sigma'$ the element of $I^{i+j}$ whose first $i$ coordinates are
those of $\sigma$ and the last $j$ coordinates are those of $\sigma'$. For all
$i\in\N$ and $\sigma\in I^*\cup I^\N,$ denote by
$\sigma|_i$ the word in $I^i$ formed by the first $i$ elements of $\sigma$. For
$\sigma\in I^*$ and $\tau\in I^*\cup I^\N$, we write $\sigma\prec\tau$ if
the sequence $\tau$ starts with $\sigma$. 

Let $\Omega$ be the set of 
functions $\omega\colon I^*\to\{c,n\}$ equipped with the topology induced by 
the metric $\rho(\omega,\omega')=k^{-|\omega\wedge\omega'|}$, where
\[
|\omega\wedge\omega'|=\min\{j\in\N\mid\exists \sigma\in I^j
  \text{ with }\omega(\sigma)\ne\omega'(\sigma)\}.
\]
Each $\omega\in\Omega$ can be thought of as a code
that tells us which cubes we choose (c) and which we neglect (n). More
precisely, let $\omega\in\Omega$. We start with the unit cube $[0,1]^m$ and 
denote it by $J_\emptyset$. We divide $J_\emptyset$ into $k^m$ closed $k$-adic 
cubes with side length $k^{-1}$, enumerate them with letters from alphabet $I$ 
and repeat this procedure inside each subcube. For all 
$\sigma\in I^i$, we use the notation $J_\sigma$ for the unique closed 
subcube of $J_\emptyset$ with side length $k^{-i}$ coded by $\sigma$.
The image of $\eta\in I^\N$ under the natural projection from 
$I^\N$ to $[0,1]^m$ is denoted by $x(\eta)$, that is, 
\[
x(\eta)=\bigcap_{i=0}^\infty J_{\eta|_i},
\]
where $\eta|_0=\emptyset$. If $\omega(\sigma)=n$ for $\sigma\in I^i$ then 
$J_\sigma(\omega)=\emptyset$, and if $\omega(\sigma)=c$ then
$J_\sigma(\omega)=J_\sigma$. Define
\[
K_\omega=\bigcap_{i=0}^\infty\bigcup_{\sigma\in I^i}J_\sigma(\omega).
\]

Fix $0\le p\le 1$. We make the above construction random by demanding that 
if $J_\sigma$ is chosen then $J_{\sigma\ast j}$, $j=1,\dots,k^m$, are
chosen independently with probability $p$. Let $P$ be the natural Borel 
probability measure on $\Omega$, that is, for all $\sigma\in I^*$ and 
$j=1,\dots,k^m$, 
\[
 \begin{split}
 &P(\omega(\emptyset)=c)=1,\\
 &P(\omega(\sigma*j)=c\mid\omega(\sigma)=c)=p,\\
 &P(\omega(\sigma*j)=n\mid\omega(\sigma)=n)=1. 
 \end{split}
\]
It is a well-known result in the theory of branching processes that if the 
expectation of the number of chosen cubes of side length $k^{-1}$
does not exceed one, then the limit set $K_\omega$ is $P$-almost surely an empty
set (see \cite[Theorem 1]{AN}). In our case, this expectation equals 
$k^mp$ and, thus, with positive probability, $K_\omega\ne\emptyset$ provided 
that $k^{-m}<p\le 1$. According to \cite[Theorem 1.1]{MW} (see also \cite{KP}),
the Hausdorff dimension of $K_\omega$ is $P$-almost surely equal to
\begin{equation}\label{dim}
d=\frac{\log(k^mp)}{\log k}=m+\frac{\log p}{\log k}
\end{equation}
provided that $K_\omega\ne\emptyset$. For $P$-almost all $\omega\in\Omega$,
there exists a natural Radon measure $\nu_\omega$ on $K_\omega$ (see 
\cite[Theorem 3.2]{MW})
and, moreover, there is a natural Radon probability measure $Q$ 
on $I^\N\times\Omega$ such that, for every Borel set 
$B\subset I^\N\times\Omega$, we have
\begin{equation}\label{Qdef}
Q(B)=\frac 1{(\diam J_\emptyset)^d}\int\mu_\omega(B_\omega)dP(\omega),
\end{equation}
where $B_\omega=\{\eta\in I^\N\mid (\eta,\omega)\in B\}$, $\nu_\omega$ is the 
image of $\mu_\omega$ under the natural projection and $\diam A$ is the diameter 
of a set $A$ (see \cite[(1.13)]{GMW}).

We denote by $\card A$ the number of elements in a set $A$.
For a word $\sigma\in I^*$, consider the martingale
$\{N_{j,\sigma}k^{-jd}\}_{j\in\N}$, where 
\[
N_{j,\sigma}=\card\{\tau\in I^*\mid |\tau|=|\sigma|+j,\tau\succ\sigma 
\text{ and }\omega(\tau)=c\},
\] 
and denote its almost sure (finite) limit by $X_\sigma(\omega)$.
For all $l\in\N\cup\{0\}$, define a random variable $X_l$ on
$I^\N\times\Omega$ by $X_l(\eta,\omega)=X_{\eta|_l}(\omega)$.
It is easy to see that, for $P$-almost all $\omega\in\Omega$,
\[ 
 X_\sigma(\omega)=\sum_{\tau\in I^j}k^{-jd}X_{\sigma*\tau}(\omega)
 \1_{\{\omega(\sigma*\tau)=c\}}
\]
for all $j\in\N$, where the characteristic function of a set $A$ is 
denoted by $\1_A$. Further, for $\sigma,\tau\in I^*$, the random
variables $X_\sigma$ and $X_\tau$ are identically 
distributed (see \cite[Proposition 1]{Be}) and, if 
$\sigma\not\prec\tau$ and $\tau\not\prec\sigma$, they are independent. 
Thus, $X_l$, $l\in\N\cup\{0\}$, have the same distribution.
According to \cite[Theorem 3.2]{MW}, the variables $X_\sigma(\omega)$ are related
to the measure $\nu_\omega$ for $P$-almost all $\omega\in\Omega$ by the formulae
\begin{align}
&\nu_\omega(J_\sigma)=(\diam J_\sigma)^dX_\sigma(\omega)\text{ for all }\sigma\in I^*
\text{ and}\label{nuversusX}\\
&\sum_{\substack{\tau\in I^j\\J_\tau\cap B\ne\emptyset}}l_\tau^dX_\tau(\omega)
\searrow\nu_\omega(B)\text{ as }j\to\infty\text{ for all Borel sets }B\subset K_\omega,\label{nudef}
\end{align}
where $l_\tau=\diam J_\tau=\diam J_\emptyset\,k^{-j}\1_{\{\omega(\tau)=c\}}$.

By \eqref{Qdef} and \eqref{nuversusX}, expectations with respect to the 
measures $P$ and $Q$ are connected in the following way (see also 
\cite[(1.16)]{GMW}): if $j\in\N$ and
$Y\colon I^\N\times\Omega\to\R$ is a random variable such that
$Y(\eta,\omega)=Y(\eta^\prime,\omega)$ provided that $\eta|_j=\eta^\prime|_j$, then
\begin{equation}\label{expectationqtop}
 E_Q[Y]=E_P\Bigl[\sum_{\substack{\sigma\in I^j\\
  \omega(\sigma)=c}} k^{-jd}X_\sigma Y(\sigma,\cdot)\Bigr].
\end{equation}
Hence, we have 
\begin{equation}\label{expectationXzero}
Q(X_l=0)=0\text{ and }E_Q[X_l]=E_P[X_0^2]<\infty
\end{equation}
for all $l\in\N\cup\{0\}$ (see \cite[Theorem 2.1]{MW}).

\section{Porosities}\label{porosities}

In this section, we define porosities and mean porosities of sets and  
measures and prove some basic properties for them.

\begin{definition}\label{defporoset}
Let $A\subset\R^m$, $x\in\R^m$ and $r>0$. The local 
porosity of $A$ at $x$ at distance $r$ is 
\[
 \begin{split}
  \por(A,x,r)=\sup\{\alpha\ge 0\mid\, &\text{there is }z\in\R^n
    \text{ such that }\\
  &B(z,\alpha r)\subset B(x,r)\setminus A\},
 \end{split}
\]
where the open ball centred at $x$ and with radius $r$ is denoted by $B(x,r)$.
The lower and upper porosities of $A$ at $x$ are defined as
\[
\lpor(A,x)=\liminf_{r\to 0}\por(A,x,r)\text{ and }
\upor(A,x)=\limsup_{r\to 0}\por(A,x,r),
\]
respectively. If $\lpor(A,x)=\upor(A,x)$, the common value, denoted by 
$\por(A,x)$, is called the porosity of $A$ at $x$.
\end{definition}

\begin{definition}\label{defporomeas}
The lower and upper porosities of a Radon measure $\mu$ on $\R^m$
at a point $x\in\R^m$ are defined by
\[
 \begin{split}
 &\lpor(\mu,x)=\lim_{\varepsilon\to 0}\liminf_{r\to 0}\por(\mu,x,r,\varepsilon)
   \text{ and}\\ 
 &\upor(\mu,x)=\lim_{\varepsilon\to 0}\limsup_{r\to 0}
   \por(\mu,x,r,\varepsilon),
 \end{split}
\]
respectively, where for all $r,\varepsilon>0$,
\[
 \begin{split}
  \por(\mu,x,r,\varepsilon)
   =&\sup\{\alpha\ge 0\mid\text{ there is }z\in\R^m\text{ such that }\\
   &B(z,\alpha r)\subset B(x,r)\text{ and }\mu(B(z,\alpha r))
    \le\varepsilon\mu(B(x,r))\}.
 \end{split}
\]
If the upper and lower porosities agree, the common value is called the porosity
of $\mu$ at $x$ and denoted by $\por(\mu,x)$.
\end{definition}

\begin{remark}\label{difdef}
(a) In some sources, the condition $B(z,\alpha r)\subset B(x,r)\setminus A$
in Definition~\ref{defporoset} is replaced by the condition 
$B(z,\alpha r)\cap A=\emptyset$, leading to the definition 
\[
 \begin{split}
  \widetilde{\por}(A,x,r)=\sup\{\alpha\ge 0\mid\, &\text{there is }z\in B(x,r)
   \text{ such that }\\
  &B(z,\alpha r)\cap A=\emptyset\}.
 \end{split}
\]
It is not difficult to see that
\[
\widetilde{\por}(A,x)=\frac{\por(A,x)}{1-\por(A,x)},
\]
which is valid both for the lower and upper porosity. Indeed, this follows
from two geometric observations. First, $B(z,\alpha r)\cap A=\emptyset$ 
with $z\in\partial B(x,r)$ if and only if 
$B(z,\alpha r)\subset B(x,(1+\alpha)r)\setminus A$ with 
$\partial B(z,\alpha r)\cap\partial B(x,(1+\alpha)r)\ne\emptyset$, where 
the boundary of a set $B$ is denoted by $\partial B$. Second, at local minima 
and maxima of the function $r\mapsto\por(A,x,r)$, we have
$\partial B(z,\alpha r)\cap\partial B(x,(1+\alpha)r)\ne\emptyset$, and at  
local minima and maxima of the function $r\mapsto\widetilde\por(A,x,r)$, we 
have $z\in\partial B(x,r)$.

(b) Unlike the dimension, the porosity is sensitive to the metric. For 
example, defining cube-porosities by using cubes instead of balls in the 
definition, there is no formula to convert porosities to cube-porosities or vice
versa. It is easy to construct a set
such that the cube-porosity attains its maximum value (at some point) but the 
porosity does not. Take, for example, the union of the x- and y-axes in the 
plane. However, the lower porosity is positive, if and only if the lower
cube-porosity is positive. 

(c) In general metric spaces, in addition to 
$B(z,\alpha r)\subset B(x,r)\setminus A$, it is sometimes useful to require 
that the empty ball $B(z,\alpha r)$ is inside the reference ball $B(x,r)$ also 
algebraically, that is, $d(x,z)+\alpha r\le r$. For further discussion about
this matter, see \cite{MMPZ}.
\end{remark}

The lower and upper porosities give the relative sizes of the largest and 
smallest holes, respectively. Taking into considerations the frequency of 
scales where the holes appear, leads to the notion of mean porosity. We proceed
by giving a definition which is adapted to the Mandelbrot percolation process. 
We will use the maximum metric $\varrho$, that is, 
$\varrho(x,y)=\max_{i\in\{1,\dots,m\}}\{|x_i-y_i|\}$, and denote by $B_\varrho(y,r)$ 
the open ball centred at $y$ and with radius $r$ with respect to this metric. 
Recall that the balls in the maximum metric are cubes whose 
faces are parallel to the coordinate planes.

\begin{definition}\label{ourmean}
Let $A\subset\R^m$, $\mu$ be a Radon measure on $\R^m$, $x\in\R^m$, 
$\alpha\in[0,1]$ and $\varepsilon>0$. For $j\in\N$, we say that $A$ has an 
$\alpha$-hole at scale $j$ near $x$ if there is a point $z\in Q_j^k(x)$ 
such that 
\[
B_\varrho(z,\tfrac 12\alpha k^{-j})\subset Q_j^k(x)\setminus A.
\]
Here $Q_j^k(x)$ is the half-open $k$-adic cube of side length $k^{-j}$ 
containing $x$ and $B_\varrho(z,\tfrac 12\alpha k^{-j})$ is called
an $\alpha$-hole. We say that $\mu$ has an $(\alpha,\varepsilon)$-hole at scale 
$j$ near $x$ if there is a point $z\in Q_j^k(x)$ such that 
\[
B_\varrho(z,\tfrac 12\alpha k^{-j})\subset Q_j^k(x)\text{ and }
\mu(B_\varrho(z,\tfrac 12\alpha k^{-j}))\le\varepsilon\mu(Q_j^k(x)).
\]
\end{definition}

\begin{remark}\label{whytwo}
Note that, unlike in Definition \ref{defporoset}, we have divided the radius of
the ball in the complement of $A$ as well as that with small measure
by $2$ and, therefore, $\alpha$ may attain values between $0$ and $1$. The 
reason for this is that the point $x$ may be arbitrarily close to the boundary
of $Q_j^k(x)$ and, if the whole cube $Q_j^k(x)$ is empty, it is natural to say 
that there is a hole of relative size $1$.  
\end{remark}

\begin{definition}\label{defmean}
Let $\alpha\in[0,1]$. The lower $\alpha$-mean 
porosity of a set $A\subset\R^m$ at a point $x\in\R^m$ is
\[
\underline{\kappa}(A,x,\alpha)=\liminf_{i\to\infty}\frac{N_i(A,x,\alpha)}i
\]
and the upper $\alpha$-mean porosity is 
\[ 
\overline{\kappa}(A,x,\alpha)=\limsup_{i\to\infty}\frac{N_i(A,x,\alpha)}i,
\]
where 
\[
 \begin{split} 
  N_i(A,x,\alpha)=\card\{j\in\N\mid\, & j\le i\text{ and }A\text{ has  
    an }\alpha\text{-hole at scale }\\
  &j\text{ near }x\}.
 \end{split}
\]
In the case the limit exists, it is called the $\alpha$-mean porosity and
denoted by $\kappa(A,x,\alpha)$. The lower $\alpha$-mean porosity of a Radon 
measure $\mu$ on $\R^m$ at $x\in\R^m$ is 
\[
\underline{\kappa}(\mu,x,\alpha)=\lim_{\e\to 0}\liminf_{i\to\infty}
  \frac{\widetilde N_i(\mu,x,\alpha,\e)}i
\]
and the upper one is
\[  
\overline{\kappa}(\mu,x,\alpha)=\lim_{\e\to 0}\limsup_{i\to\infty}
  \frac{\widetilde N_i(\mu,x,\alpha,\e)}i,
\]
where 
\[
 \begin{split}
  \widetilde N_i(\mu,x,\alpha,\varepsilon)=\card\{j\in\N\mid\, & j\le i
   \text{ and }\mu\text{ has an }(\alpha,\varepsilon)\text{-hole at}\\
  &\text{scale }j\text{ near } x\}.
 \end{split}
\]
If the lower and upper mean porosities coincide, the common value, denoted by 
$\kappa(\mu,x,\alpha)$, is called the $\alpha$-mean porosity of $\mu$.
\end{definition}

\begin{remark}\label{meanunstable}
Mean porosity is highly sensitive to the choice of 
parameters. The definition is given in terms of $k$-adic cubes. For the 
Mandelbrot percolation, this is natural. For general sets, fixing an integer 
$h>1$, a natural choice is to say that $A$ has an $\alpha$-hole at scale $j$ 
near $x$, if there is $z\in\R^m$ such that 
$B(z,\alpha h^{-j}r_0)\subset B(x,h^{-j}r_0)\setminus A$ for some (or for all) 
$h^{-1}<r_0\le 1$. However, the choice of $r_0$ and $h$ matters as will be shown
in Example~\ref{norel} below. Shmerkin proposed in \cite{Shm2} the following 
base and starting scale independent notion of lower mean porosity of a measure 
(which can be adapted for sets and upper porosity as well): a measure $\mu$ is 
lower $(\alpha,\kappa)$-mean porous at a point $x\in\R^m$ if
\[
\liminf_{\rho\to 1}\,(\log\tfrac 1\rho)^{-1}\int_\rho^1
  \1_{\{r\mid\por(\mu,x,r,\e)\ge\alpha\}}r^{-1}\,dr\ge\kappa\text{ for all }\e>0.
\]
The disadvantage of this definition is that it is more complicated to calculate
than the discrete version. To avoid these problems, one option
is to aim at qualitative results concerning all parameter values, as our 
approach will show.
\end{remark}

Next we give a simple example demonstrating the dependence of 
mean porosity on the starting scale and the base of scales.

\begin{example}\label{norel}
Fix an integer $h>1$. In this example, we use a modification of
Definitions~\ref{ourmean} and \ref{defmean} where $A\subset\R^m$ has an 
$\alpha$-hole at scale $j$ near $x$, if there exists 
$z\in\R^m$ such that $B(z,\alpha h^{-j})\subset B(x,h^{-j})\setminus A$. 
Let $x\in\R^2$. We define a set $A\subset\R^2$ as follows.
For all $i\in\N\cup\{0\}$, consider the half-open annulus 
$D(i)=\{y\in\R^2\mid h^{-i-1}<|y-x|\le h^{-i}\}$. Let 
$A=\bigcup_{i=0}^\infty D(3i+1)\cup D(3i+2)$, that is, we choose
two annuli out of every three successive ones and leave the third one empty.
In this case, $\kappa(A,x,\frac 12(1-h^{-1}))=\frac 13$. If we replaced $h$ by 
$h^3$ in the definition of scales, we would conclude that   
$\kappa(A,x,\frac 12(1-h^{-1}))=1$. (Note that the lower and upper 
porosities are equal.) If we define $A$ by starting with the two 
filled annuli, that is, $A=\bigcup_{i=0}^\infty D(3i)\cup D(3i+1)$, then
$\kappa(A,x,\frac 12(1-h^{-1}))=\frac 13$ using scales determined by $h$ and 
$\kappa(A,x,\frac 12(1-h^{-1}))=0$ if scales are determined by powers of $h^3$. 
By mixing these construction in a suitable way, one easily finds an example 
where $\kappa(A,x,\frac 12(1-h^{-1}))=\frac 13$ for scales given by $h$, but 
$\underline{\kappa}(A,x,\frac 12(1-h^{-1}))=0$ and 
$\overline{\kappa}(A,x,\frac 12(1-h^{-1}))=1$ if the scales are determined by 
$h^3$.
\end{example}

We finish this section with some measurability results. For that we need some 
notation.

\begin{definition}\label{holeindicator}
For all $j\in\N$ and $\alpha\in[0,1]$, define a function 
$\chi_j^\alpha\colon I^\N\times\Omega\to\{0,1\}$ by setting 
$\chi_j^\alpha(\eta,\omega)=1$, if and only if 
$K_\omega$ has an $\alpha$-hole at scale $j$ near $x(\eta)$. Define a function 
$\overline{\chi}_j^\alpha\colon I^\N\times\Omega\to\{0,1\}$ in the same 
way except that the $\alpha$-hole is a closed ball instead of an open one. 
For all $\alpha\in (0,1)$, $\e>0$ and $j\in\N$, define a function 
$\chi_j^{\alpha,\epsilon}\colon I^\N\times\Omega\to\{0,1\}$ by setting 
$\chi_j^{\alpha,\epsilon}(\eta,\omega)=1$, if and only if 
$\nu_\omega$ has an $(\alpha,\epsilon)$-hole at scale $j$ near $x(\eta)$. 
Finally, define a function 
$\overline\chi_j^{\alpha,\epsilon}\colon I^\N\times\Omega\to\{0,1\}$ by setting 
$\overline\chi_j^{\alpha,\epsilon}(\eta,\omega)=1$, if and only if there exists
$z\in Q_j^k(x(\eta))$ such that 
$\nu_\omega(\overline B_\varrho(z,\tfrac 12\alpha k^{-j}))<\varepsilon
  \nu_\omega(Q_j^k(x(\eta)))$. 
Here the closed ball in metric $\varrho$ centred at $z\in\R^m$ with radius 
$r>0$ is denoted by $\overline B_\varrho(z,r)$.
\end{definition}

\begin{lemma}\label{measurability}
The maps 
\[
\begin{split}
&(\eta,\omega)\mapsto\underline{\kappa}(K_\omega,x(\eta),\alpha),\\
&(\eta,\omega)\mapsto\overline{\kappa}(K_\omega,x(\eta),\alpha),\\
&(\eta,\omega)\mapsto\underline{\kappa}(\nu_\omega,x(\eta),\alpha)
  \text{ and}\\
&(\eta,\omega)\mapsto\overline{\kappa}(\nu_\omega,x(\eta),\alpha)
\end{split}
\]
are Borel measurable for all $\alpha\in[0,1]$. 
\end{lemma}

\begin{proof}
Note that $\overline{\chi}_j^\alpha(\cdot,\omega)$ is locally constant for all 
$\omega\in\Omega$, that is, its value depends only on $\eta|_j$.
Further, suppose that $\overline{\chi}_j^\alpha(\eta,\omega)=1$. Then $K_\omega$
has a closed $\alpha$-hole $H$ at scale $j$ near $x(\eta)$. Since $H$ 
and $K_\omega$ are closed, their distance is positive. So there exists a finite
set $T\subset I^*$ such that $\omega(\tau)=n$ for all $\tau\in T$ and
\[
H\subset\bigcup_{\tau\in T}J_\tau.
\]
If $\omega'\in\Omega$ is close to $\omega$, then 
$\omega'(\tau)=n$ for all $\tau\in T$, which implies that 
$\overline{\chi}_j^\alpha(\eta,\omega')=1$. We conclude that
$\overline{\chi}_j^\alpha$ is continuous at $(\eta,\omega)$. Trivially, 
$\overline{\chi}_j^\alpha$ is lower semi-continuous at those points where
$\overline{\chi}_j^\alpha(\eta,\omega)=0$. Therefore, $\overline{\chi}_j^\alpha$ 
is lower semi-continuous. 

Let $\alpha_i$ be a strictly increasing sequence approaching $\alpha$. We claim
that
\begin{equation}\label{alphalim}
\lim_{i\to\infty}\overline{\chi}_j^{\alpha_i}(\eta,\omega)=\chi_j^\alpha(\eta,\omega)
\end{equation}
for all $(\eta,\omega)\in I^\N\times\Omega$. Indeed, obviously
$\chi_j^\alpha(\eta,\omega)\le\overline{\chi}_j^{\alpha_i}(\eta,\omega)$ for all
$i\in\N$,
and the sequence $(\overline{\chi}_j^{\alpha_i}(\eta,\omega))_{i\in\N}$ is
decreasing. Thus, it is enough to study the case 
$\lim_{i\to\infty}\overline{\chi}_j^{\alpha_i}(\eta,\omega)=1$. Let
$(\overline B_\varrho(z_i,\frac 12\alpha_ik^{-j}))_{i\in\N}$ be a corresponding
sequence of closed holes. In this case, one may find a convergent subsequence 
of $(z_i)_{i\in\N}$ converging to $z\in\R^d$ and
$B_\varrho(z,\frac 12\alpha k^{-j})\subset Q_j^k(x(\eta))\setminus K_\omega$, 
completing the proof of \eqref{alphalim}. As a limit of semi-continuous
functions, $\chi_j^\alpha$ is Borel measurable. Now 
$N_i(K_\omega,x(\eta),\alpha)=\sum_{j=1}^i\chi_j^\alpha(\eta,\omega)$, implying 
that the map 
$(\eta,\omega)\mapsto\underline{\kappa}(K_\omega,x(\eta),\alpha)$  
(as well as the upper mean porosity) is Borel measurable. 

By construction, the map $\omega\mapsto X_\tau(\omega)$ is Borel measurable for
all $\tau\in I^*$. Therefore, $\omega\mapsto\nu_\omega(B)$ is a Borel map for
all Borel sets $B\subset\R^m$ by \eqref{nudef}. In particular, the map
$(\eta,\omega)\mapsto\nu_\omega\bigl(\overline B_\varrho(z,\tfrac 12\alpha k^{-j})
\bigr)-\varepsilon\nu_\omega\bigl(Q_j^k(x(\eta))\bigr)$
is Borel measurable for all $z\in\R^m$, $\alpha\in[0,1]$, $\e>0$ and $j\in\N$.
Let $(z_i)_{i\in\N}$ be a dense set in $[0,1]^m$. Let $s>0$, $\alpha\in[0,1]$
and $j\in\N$. Suppose that there exists $z\in Q_j^k(x(\eta))$ such that 
$\nu_\omega(\overline B_\varrho(z,\tfrac 12\alpha k^{-j}))<s$. Since the map
$x\mapsto\nu_\omega(\overline B_\varrho(x,r))$ is upper semicontinuous, there 
exists $z_i\in Q_j^k(x(\eta))$ such that
$\nu_\omega(\overline B_\varrho(z_i,\tfrac 12\alpha k^{-j}))<s$. Thus,
$\overline\chi_j^{\alpha,\e}$ is Borel measurable. Further, 
$\chi_j^{\alpha,\e}(\eta,\omega)=1$ if and only if there exist an increasing
sequence $(\alpha_i)_{i\in\N}$ tending to $\alpha$ and a decreasing sequence
$(\e_i)_{i\in\N}$ tending to $\e$ such that  
$\overline\chi_j^{\alpha_i,\e_i}(\eta,\omega)=1$. Therefore, $\chi_j^{\alpha,\e}$
is Borel measurable, and the claim follows as in the case of mean porosities
of sets.
\end{proof}

\begin{remark}\label{monotonicity}
(a) Note that, for all $(\eta,\omega)\in I^\N\times\Omega$, the function 
$\alpha\mapsto\chi_0^\alpha(\eta,\omega)$ is decreasing 
and, thus, the lower and upper mean porosity functions are also 
decreasing as functions of $\alpha$.

(b) Later, we will need modifications of the functions $\chi_j^\alpha$
defined in the proof of Lemma~\ref{measurability}. Their Borel measurability
can be proven analogously to that of $\chi_j^\alpha$.
\end{remark}

\section{Results}\label{results}

\renewcommand{\theenumi}{\roman{enumi}}

In this section, we state and prove our results concerning mean
porosities of Mandelbrot percolation and its natural measure.
To prove the existence of mean porosity and to 
compare the mean porosities of the limit set and the construction measure,
we need a tool to establish the validity of the strong law of large numbers for 
certain sequences of random variables. We will use \cite[Theorem 1]{HRV} (see 
also \cite[Corollary 11]{L}), which we state (in a simplified form) for the 
convenience of the reader.

\begin{theorem}\label{HRV}
Let $\{Y_n\}_{n\in\N}$ be a sequence of square-integrable random variables 
and suppose that there exists a sequence of constants $(\rho_k)_{k\in\N}$
such that 
\[
\sup_{n\in\N}|\Cov(Y_n,Y_{n+k})|\le\rho_k
\]
for all $k\in\N$. Assume that 
\[
\sum_{n=1}^\infty\frac{\Var(Y_n)\log^2n}{n^2}<\infty\text{ and }
  \sum_{k=1}^\infty\rho_k<\infty.
\]
Then $\{Y_n\}_{n\in\N}$ satisfies the strong law of large numbers. Here the 
covariance and variance are denoted by $\Cov$ and $\Var$, respectively.
\end{theorem}

We will apply Theorem~\ref{HRV} to stationary sequences of random variables
which are indicator functions of events with equal probabilities.
In this setup, all conditions of the theorem will be satisfied if
\begin{equation}\label{convergencecovariances}
\sum_{j=1}^\infty \Cov(Y_0,Y_j)<\infty
\end{equation}
and, in particular, if $Y_i$ is independent from $Y_j$ once
$|i-j|$ is greater than some fixed integer. 

For all $\alpha\in[0,1]$ and $j,r\in\N$, define 
$\chi_{j,r}^\alpha\colon I^\N\times\Omega\to\{0,1\}$ similarly to 
$\chi_j^\alpha$ with the exception that the whole hole is assumed to be in 
$J_{\eta|_j}\setminus J_{\eta|_{j+r}}$. Observe that 
$\chi_{j,r}^\alpha(\eta,\omega)=\chi_{j,r}^\alpha(\eta',\omega)$ provided that
$\eta|_{j+r}=\eta'|_{j+r}$. Therefore, for any $\tau\in I^{j+r}$, we may define
$\chi_{j,r}^\alpha(\tau,\omega)=\chi_{j,r}^\alpha(\eta,\omega)$, where 
$\eta|_{j+r}=\tau$. Note that, given $\omega(\tau|_j)=c$, the value of the
function $\chi_{j,r}^\alpha(\tau,\cdot)$ depends only on the restriction of 
$\omega$ to $J_{\tau|_j}\setminus J_{\tau|_{j+r}}$. 

\begin{lemma}\label{independence}
Let $\alpha\in[0,1]$ and $r\in\N$. The random variables $\chi_{j,r}^\alpha$
and $\chi_{i,r}^\alpha$ are $Q$-independent for all $i,j\in\N$ with
$|i-j|\ge r$ and 
\[
\Cov(\chi_{n,r}^\alpha,\chi_{n+k,r}^\alpha)=\Cov(\chi_{0,r}^\alpha,\chi_{k,r}^\alpha)
\]
for all $n,k,r\in\N$.
\end{lemma}

\begin{proof}
Observe that, for any $\tau\in I^{j+r}$, the variables $X_\tau$ and 
$\chi_{0,r}^\alpha(\tau,\cdot)$ are $P$-independent given $\omega(\tau)=c$. 
Recalling that $E_P[X_\tau\mid\omega(\tau)=c]=1$ by \eqref{Qdef} and 
\eqref{nuversusX}, we obtain by \eqref{expectationqtop} that
\begin{align*}
E_Q[\chi_{0,r}^\alpha]&=E_P\bigl[\sum_{\tau\in I^r}
k^{-rd}\1_{\{\omega(\tau)=c\}}X_\tau
  \chi_{0,r}^\alpha(\tau,\cdot)\bigr]\\
 &=E_P\bigl[\sum_{\tau\in I^r}k^{-rd}\1_{\{\omega(\tau)=c\}}
  E_P[\chi_{0,r}^\alpha(\tau,\cdot)\mid\omega(\tau)=c]\bigr].
\end{align*}
Further, for any $j\in\N$ (using the above calculation in the third equality),
\begin{align*}
&E_Q[\chi_{j,r}^\alpha]=E_P\bigl[\sum_{\sigma\in I^j}k^{-jd}\1_{\{\omega(\sigma)=c\}}
  \sum_{\tau\in I^r}k^{-rd}\1_{\{\omega(\sigma\ast\tau)=c\}}X_{\sigma\ast\tau}
  \chi_{j,r}^\alpha(\sigma\ast\tau,\cdot)\bigr]\\
 &=E_P\Bigl[\sum_{\sigma\in I^j}k^{-jd}\1_{\{\omega(\sigma)=c\}}
  E_P\bigl[\sum_{\tau\in I^r}k^{-rd}\1_{\{\omega(\sigma\ast\tau)=c\}}X_{\sigma\ast\tau}
  \chi_{j,r}^\alpha(\sigma\ast\tau,\cdot)\mid\omega(\sigma)=c\bigr]\Bigr]\\
 &=E_P\bigl[\sum_{\sigma\in I^j}k^{-jd}\1_{\{\omega(\sigma)=c\}}E_Q[\chi_{0,r}^\alpha]
  \bigr]=E_Q[\chi_{0,r}^\alpha].
\end{align*}
Let $i,j\in\N$ be such that $j-i\ge r$. Then (using the above calculation in 
the third and fourth equality) 
\begin{align*}
&E_Q[\chi_{i,r}^\alpha\chi_{j,r}^\alpha]\\
&=E_P\bigl[\sum_{\tau\in I^{i+r}}k^{-(i+r)d}\1_{\{\omega(\tau)=c\}}
 \chi_{i,r}^\alpha(\tau,\cdot)
 \sum_{\sigma\in I^{j-i}}k^{-(j-i)d}\1_{\{\omega(\tau\ast\sigma)=c\}}X_{\tau\ast\sigma}
 \chi_{j,r}^\alpha(\tau\ast\sigma,\cdot)\bigr]\\
&=E_P\Bigl[\sum_{\tau\in I^{i+r}}k^{-(i+r)d}\1_{\{\omega(\tau)=c\}}
 E_P[\chi_{i,r}^\alpha(\tau,\cdot)\mid\omega(\tau)=c]\\
&\phantom{llllll}\times E_P\bigl[
 \sum_{\sigma\in I^{j-i}}k^{-(j-i)d}X_{\tau\ast\sigma}\1_{\{\omega(\tau\ast\sigma)=c\}}
 \chi_{j,r}^\alpha(\tau\ast\sigma,\cdot)\mid\omega(\tau)=c\bigr]\Bigr]\\
&=E_P\bigl[\sum_{\tau\in I^{i+r}}k^{-(i+r)d}\1_{\{\omega(\tau)=c\}}
 E_P[\chi_{i,r}^\alpha(\tau,\cdot)\mid\omega(\tau)=c] E_Q[\chi_{0,r}^\alpha]
 =E_Q[\chi_{i,r}^\alpha]E_Q[\chi_{j,r}^\alpha].
\end{align*}
The last claim follows from a similar calculation.
\end{proof}

Next we prove a lemma which gives lower and upper bounds for mean porosities
at typical points.

\begin{lemma}\label{holecontinuity} For all $\alpha\in(0,1)$, we have
\[
E_Q[\overline{\chi}_0^\alpha] \le
\underline{\kappa}(K_\omega,x(\eta),\alpha)
\le\overline{\kappa}(K_\omega,x(\eta),\alpha)
\le E_Q[ \chi_0^\alpha]
\]
for $Q$-almost all $(\eta,\omega)\in I^{\N}\times\Omega$.
\end{lemma}

\begin{proof}
Note that for every $\alpha\in(0,1)$ and $r\in\N$ with $k^{-r}<\alpha$, we have
\[
 \chi_{j,r}^\alpha(\eta,\omega)\le\chi_j^\alpha(\eta,\omega)\le 
  \chi_{j,r}^{\alpha-k^{-r}}(\eta,\omega)
\]
for all $(\eta,\omega)\in I^{\N}\times\Omega$ satisfying $x(\eta)\in K_\omega$.
Recall that $\nu_\omega$ is supported on $K_\omega$ for $P$-almost all 
$\omega\in\Omega$. 
Combining Lemma~\ref{independence} and Theorem~\ref{HRV}, we conclude that, 
for all $r\in\N$ and $Q$-almost all $(\eta,\omega)\in I^{\N}\times\Omega$,
we have
\begin{align*}
E_Q[\chi_{0,r}^\alpha]&=\lim_{n\to\infty}\frac 1n \sum_{j=1}^n 
  \chi_{j,r}^\alpha(\eta,\omega)\le\underline\kappa(K_\omega,x(\eta),\alpha)\\
 &\le\overline\kappa(K_\omega,x(\eta),\alpha)\le\lim_{n\to\infty}\frac 1n 
  \sum_{j=1}^n \chi_{j,r}^{\alpha-k^{-r}}(\eta,\omega)=E_Q[\chi_{0,r}^{\alpha-k^{-r}}].
\end{align*}
Observe that, for all $(\eta,\omega)\in I^{\N}\times\Omega$ satisfying 
$x(\eta)\in K_\omega$, we have 
$\lim_{r\to\infty}\chi_{0,r}^\alpha(\eta,\omega)
   \ge\overline\chi_0^\alpha(\eta,\omega)$, since the distance between a closed
$\alpha$-hole and $K_\omega$ is positive. Further, the inequality
$\chi_{0,r}^{\alpha-k^{-r}}\le\overline{\chi}_0^{\alpha-2k^{-r}}$ is always valid and
$\lim_{r\to\infty}\overline{\chi}_0^{\alpha-2k^{-r}}=\chi_0^\alpha$ by \eqref{alphalim}.
Hence,
\[
E_Q[\overline{\chi}_0^\alpha]\le\underline\kappa(K_\omega,x(\eta),\alpha)
\le\overline\kappa(K_\omega,x(\eta),\alpha)\le E_Q[ \chi_0^\alpha]
\]
for $Q$-almost all $(\eta,\omega)\in I^{\N}\times\Omega$.
\end{proof}

In fact, the upper bound we have found is an exact equality.

\begin{proposition}\label{exactupper}
For all $\alpha\in (0,1)$, we have that
\[
\overline{\kappa}(K_\omega,x(\eta),\alpha)=E_Q[ \chi_0^\alpha]
\]
for $Q$-almost all $(\eta,\omega)\in I^{\N}\times\Omega$. 
\end{proposition}

\begin{proof}
We start by proving that, for all $\alpha\in (0,1)$,
\[
\lim_{j\to\infty}\Cov(\chi_0^\alpha,\chi_j^\alpha)=0.
\]
Let $\alpha\in (0,1)$. Note that, for all $(\eta,\omega)\in I^{\N}\times\Omega$
satisfying $x(\eta)\in K_\omega$, we have 
\[
\chi_0^\alpha(\eta,\omega)\le\chi_{0,r}^{\alpha-k^{-r}}(\eta,\omega)
\le\overline{\chi}_0^{\alpha-2k^{-r}}(\eta,\omega)
\]
for all $r\in\N$ such that $2k^{-r}<\alpha$. Therefore, for all $j\in\N$ with
$2k^{-j}<\alpha$, we have the following estimate
\[
\Cov(\chi_0^\alpha,\chi_j^\alpha)=E_Q[\chi_0^\alpha\chi_j^\alpha]
-E_Q[\chi_0^\alpha]E_Q[\chi_j^\alpha]\le E_Q[\chi_{0,j}^{\alpha-k^{-j}}\chi_j^\alpha]-E_Q[\chi_0^\alpha]^2.
\]
Next we note that the random variables $\chi_{0,j}^{\alpha-k^{-j}}$ and 
$\chi_j^\alpha$ are $Q$-independent (compare Lemma~\ref{independence}),
hence
\[
\Cov(\chi_0^\alpha,\chi_j^\alpha)\le E_Q[\chi_0^\alpha]E_Q[\chi_{0,j}^{\alpha-k^{-j}}
  -\chi_0^\alpha]
\le E_Q[\chi_0^\alpha]E_Q[\overline{\chi}_0^{\alpha-2k^{-j}}-\chi_0^\alpha] .
\]
Since $\lim_{j\to\infty}(\alpha-2k^{-j})=\alpha$, the equality \eqref{alphalim} 
and the dominated convergence theorem imply that 
$\lim_{j\to\infty}E_Q[\overline{\chi}_0^{\alpha-2k^{-j}}-\chi_0^\alpha]=0$. Now, by 
Bernstein's theorem, the sequence 
$\frac 1n N_n(A,x,\alpha)=\frac 1n\sum_{i=1}^n\chi_i^\alpha$
converges in probability to $E_Q[\chi_0^\alpha]$. Once we have the convergence 
in probability, we can find a subsequence converging almost surely and,
therefore, the upper bound in Theorem~\ref{holecontinuity} is attained.
\end{proof}

Define
\[
D=\{\alpha\in (0,1)\mid\beta\mapsto E_Q[\chi_0^\beta]\text{ is discontinuous at }
  \beta=\alpha\}.
\]
Since $\beta\mapsto E_Q[\chi_0^\beta]$ is decreasing, the set $D$ is countable.

\begin{theorem}\label{meanporosityexists}
For $Q$-almost all $(\eta,\omega)\in I^{\N}\times\Omega$, we have
\[
\kappa(K_\omega,x(\eta),\alpha)=E_Q[ \chi_0^\alpha]
\]
for all $\alpha\in(0,1)\setminus D$. In particular, for $Q$-almost all 
$(\eta,\omega)\in I^{\N}\times\Omega$, the function 
$\alpha\mapsto\kappa(K_\omega,x(\eta),\alpha)$ is defined and continuous at all
$\alpha\in(0,1)\setminus D$.
\end{theorem}

\begin{proof}
Since $\chi_0^{\alpha'}\le\overline\chi_0^\alpha\le\chi_0^\alpha$ for all 
$\alpha'>\alpha$, we have that $E_Q[\overline\chi_0^\alpha]=E_Q[\chi_0^\alpha]$ for
all $\alpha\in (0,1)\setminus D$. Lemma~\ref{holecontinuity} implies that, for 
all $\alpha\in (0,1)\setminus D$, there exists a Borel set 
$B_\alpha\subset I^{\N}\times\Omega$ such that 
$\kappa(K_\omega,x(\eta),\alpha)=E_Q[ \chi_0^\alpha]$ for all 
$(\eta,\omega)\in B_\alpha$ and $Q(B_\alpha)=1$. Let $(\alpha_i)_{i\in\N}$ be a 
dense set in $(0,1)$. Since the functions  
$\alpha\mapsto\underline \kappa(K_\omega,x(\eta),\alpha)$ and
$\alpha\mapsto\overline\kappa(K_\omega,x(\eta),\alpha)$ are decreasing, we have
for all $(\eta,\omega)\in\bigcap_{i=1}^\infty B_{\alpha_i}$ that 
$\kappa(K_\omega,x(\eta),\alpha)=E_Q[ \chi_0^\alpha]$ for all 
$\alpha\in(0,1)\setminus D$. Since $Q(\bigcap_{i=1}^\infty B_{\alpha_i})=1$, the 
proof is complete.
\end{proof}

\begin{proposition}\label{comparison}
Suppose that $m=2$ and $p>k^{-1}$. Then the set $D$ is non-empty. 
\end{proposition}

\begin{proof}
Since $E_Q[\chi_0^{\alpha'}]\le E_Q[\overline\chi_0^\alpha]\le E_Q[\chi_0^\alpha]$
for all $\alpha'>\alpha$, it is enough to show that there exists 
$\alpha\in (0,1)$ such that $E_Q[\overline\chi_0^\alpha]<E_Q[\chi_0^\alpha]$.
This, in turn, follows if 
\begin{equation}\label{positivemeasure}
Q(\{(\eta,\omega)\in I^\N\times\Omega\mid\overline\chi_0^\alpha(\eta,\omega)=0
 \text{ and }\chi_0^\alpha(\eta,\omega)=1\})>0,
\end{equation}
since $\overline\chi_0^\alpha\le\chi_0^\alpha$.

It is shown in \cite{FG} that, if $m=2$ and $p>k^{-1}$, the projection of 
$K_\omega$ onto the $x$-axis is the  whole unit interval $[0,1]$ with positive 
probability. In particular, $K_\omega$ intersects all the faces of $[0,1]^2$ 
with positive probability. Let $\alpha=k^{-r}$ for some $r\in\N$. Fix 
$\sigma\in I^r$. Then there exists a Borel set $B\subset\Omega$ with $P(B)>0$
such that, for all $\omega\in B$, we have $\omega(\sigma)=n$ and  
$K_\omega$ intersects all the faces of $J_\tau$ for all $\tau\in I^{r+1}$ with
$\tau|_r\ne\sigma$. In this case, $\chi_0^\alpha(\eta,\omega)=1$ and 
$\overline\chi_0^\alpha(\eta,\omega)=0$ for all $\omega\in B$ for 
$\mu_\omega$-almost all $\eta\in B_\omega$. This implies inequality 
\eqref{positivemeasure}.
\end{proof}

\begin{remark}\label{otherdiscontinuities}
A similar construction as in the proof of Propostion~\ref{positivemeasure} can 
be done for any positive $\alpha=\sum_{j=1}^nq_jk^{-r_j}<1$, where $r_j\in\N$ and 
$q_j\in\mathbb Z$, that is, for any hole which is a finite union of 
construction squares. We do not know whether $\kappa(K_\omega,x(\eta),\alpha)$ 
exists for $\alpha\in D$.
\end{remark}

\begin{corollary}\label{positive}
For $P$-almost all $\omega\in\Omega$ and for 
$\nu_\omega$-almost all $x\in K_\omega$, we have that
\[
0<\underline{\kappa}(K_\omega,x,\alpha)\le \overline{\kappa}(K_\omega,x,\alpha)<1
\]
for all $\alpha\in (0,1)$, $\kappa(K_\omega,x,0)=1$ and $\kappa(K_\omega,x,1)=0$. 
\end{corollary} 

\begin{proof}  
Since $0<E_Q[\chi_0^\alpha]<1$ for all $\alpha\in (0,1)$ and the functions
$\alpha\mapsto\underline\kappa(K_\omega,x(\eta),\alpha)$ and 
$\alpha\mapsto\overline\kappa(K_\omega,x(\eta),\alpha)$ are decreasing,
the first claim follows from Theorem~\ref{meanporosityexists}.
The claim $\kappa(K_\omega,x,0)=1$ is obvious. Finally, if  
$\overline\kappa(K_\omega,x,1)>0$, $K_\omega$ has a $1$-hole 
near $x$ at scale $j$ for some $j\in\N$. Hence, $x$ should be on 
the boundary of the hole and $J_{\eta|_j}$ which, in turn, implies 
that $K_\omega$ has a $1$-hole near $x$ at all scales larger than $j$. Thus 
$\kappa(K_\omega,x,\alpha)=1$ for all $\alpha\le 1$ which is a contradiction
with the first claim.
\end{proof}

To study the mean porosities of the natural measure, we need some auxiliary
results. 

\begin{proposition}\label{LLN}
For all $s>0$, the sequence $\{\1_{\{X_j\le s\}}\}_{j\in\N}$
satisfies the strong law of large numbers.
\end{proposition}

\begin{proof}
Since the sequence $(X_j)_{j\in\N}$ is stationary, we only have to check that the
series \eqref{convergencecovariances} converges with $Y_j=\1_{\{X_j\le s\}}$.
Since $X_j$ and $X_0-k^{-jd}X_j$ are $Q$-independent (compare 
Lemma~\ref{independence} or see the remark before \cite[Lemma 10]{Be2}), 
recalling that $X_j$ 
and $X_0$ have the same distribution, we can make the following estimate
\[
\begin{split}
\Cov&(\1_{\{X_0\le s\}},\1_{\{X_j\le s\}})\\
&=Q(X_0\le s\text{ and }X_j\le s)-Q(X_0\le s)Q(X_j\le s)\\
&\le Q(X_0-k^{-jd}X_j\le s\text{ and }X_j\le s)-Q(X_0\le s)Q(X_j\le s)\\
&=Q(X_j\le s)\bigl(Q(X_0-k^{-jd}X_j\le s)-Q(X_0\le s)\bigr)\\
&=Q(X_0\le s)Q(s<X_0\le s+k^{-jd}X_j)\\
&\le Q(X_0\le s)\bigl(Q(s<X_0\le s+k^{-\frac 12jd})+Q(X_0>k^{\frac 12jd})\bigr).\\
\end{split}
\] 
By a result of Dubuc and Seneta \cite{DS} (see also \cite[Theorem II.5.2]{AN}), 
the distribution of $X_0$ has a continuous $P$-density $q(x)$ on $(0,+\infty)$. 
From formula \eqref{expectationqtop}, we obtain
\[
\begin{split}
Q(s<X_0\le s+k^{-\frac 12jd})&=E_Q\bigl[\1_{\{s<X_0\le s+k^{-\frac 12jd}\}}\bigr]\\
&=E_P\bigl[X_0\1_{\{s< X_0\le s+k^{-\frac 12jd}\}}\bigr]\\
&\le (s+k^{-\frac 12d})P(s< X_0\le s+k^{-\frac 12jd})\\
&\le (s+k^{-\frac 12d})k^{-\frac 12jd} \max\limits_{x\in[s,s+k^{-\frac 12d}]}q(x).\\
\end{split}
\]
Therefore, by Markov's inequality,
\begin{multline*}
\sum_{j=1}^\infty \Cov(\1_{\{X_0\le s\}},\1_{\{X_j\le s\}})\le\\
Q(X_0\le s)\bigl((s+k^{-\frac 12d})\max\limits_{x\in[s,s+k^{-\frac 12d}]}q(x)
  +E_Q(X_0)\bigr)\sum_{j=1}^\infty k^{-\frac 12jd}<\infty.
\end{multline*}
\end{proof}

For all $\alpha\in (0,1)$, $\e,\delta>0$ and $j\in\N$, define a function
$H_j^{\alpha,\e,\delta}\colon I^{\N}\times\Omega\to\{0,1\}$ by setting
$H_j^{\alpha,\e,\delta}(\eta,\omega)=1$, if and only if $\nu_\omega$ has an 
$(\alpha,\epsilon)$-hole at scale $j$ near $x(\eta)$ but $K_\omega$ does not 
have an $(\alpha-\delta)$-hole at scale $j$ near $x(\eta)$.

\begin{lemma}\label{holedif}
Let $\alpha\in (0,1)$. For all $\delta>0$, there exists $\e_0>0$ such that, for
$Q$-almost all $(\eta,\omega)\in I^{\N}\times\Omega$, we have 
\[
\limsup_{n\to\infty}\frac 1n \sum_{j=1}^n H_j^{\alpha,\e,\delta}(\eta,\omega)\le\delta
\]
for all $0<\e\le\e_0$.
\end{lemma}

\begin{proof}
Fix $0<\delta<\alpha$. Let $r\in\N$ be the smallest integer such that 
$2k^{-r}<\delta$. Let $\e>0$. Assume that $H_j^{\alpha,\e,\delta}(\eta,\omega)=1$ 
and denote by $H$ the $(\alpha,\e)$-hole at scale $j$ near $x(\eta)$. 
Considering the relative positions of $H$ and $J_{\eta|_{j+r}}$, we will argue 
that we arrive at the following possibilities:
\begin{enumerate}
\item If $J_{\eta|_{j+r}}\subset H$, we have 
  $\nu_\omega(J_{\eta|_{j+r}})\le\e\nu_\omega(J_{\eta|_j})$.
\item In the case $J_{\eta|_{j+r}}\not\subset H$, there exists $\tau_j\in I^{j+r}$
  such that $\tau_j\ne\eta|_{j+r}$, $J_{\tau_j}\subset H$, 
  $K_\omega\cap J_{\tau_j}\ne\emptyset$ and 
  $\nu_\omega(J_{\tau_j})\le\e\nu_\omega(J_{\eta|_j})$.
\end{enumerate}
Suppose that (i) is not valid. Since $H_j^{\alpha,\e,\delta}(\eta,\omega)=1$, the 
set $K_\omega$ does not have an $(\alpha-\delta)$-hole at scale $j$ near 
$x(\eta)$. Observe that 
\[
H\setminus\bigcup_{\substack{\sigma\in I^{j+r}\\ J_\sigma\not\subset H}}J_\sigma
\]
contains a cube with side length $(\alpha-\delta) k^{-j}$ since $2k^{-r}<\delta$.
Since $J_{\eta|_{j+r}}\not\subset H$, there exists $\tau_j\in I^{j+r}$ as in (ii).

Next we estimate how often (i) or (ii) may happen for $Q$-typical 
$(\eta,\omega)\in I^\N\times\Omega$. We start by 
considering the case (i). We denote by $A_j^{1,\e}$ the 
event that $\nu_\omega(J_{\eta|_{j+r}})\le\e\nu_\omega(J_{\eta|_j})$, that is,
according to  \eqref{nuversusX},
\begin{multline*}
A_j^{1,\e}=\bigl\{(\eta,\omega)\in I^\N\times\Omega\mid \\
   X_{\eta|_{j+r}}(\omega)\le
	\frac {\e}{1-\e}\sum_{\substack{\tau\in I^{j+r}\\ 
	 \eta|_j\prec\tau,\ \tau\ne\eta|_{j+r},\ \omega(\tau)=c}}X_\tau(\omega)\bigr\}.
\end{multline*}
For all $s>0$, let
\[
\begin{split}
 A_{j,1}^{s,\e}&=\bigl\{(\eta,\omega)\in I^N\times\Omega\mid 
     X_{\eta|_{j+r}}(\omega)\le \frac{\e s}{1-\e}\bigr\}\\ 
\text{ and }A_{j,2}^s&=\bigl\{(\eta,\omega)\in I^N\times\Omega\mid 
   \sum_{\substack{\tau\in I^{j+r}\\\eta|_j\prec\tau,\ \tau\ne\eta|_{j+r},\
   \omega(\tau)=c}}X_\tau(\omega)>s\bigr\}.
 \end{split}
\]

In the case (ii), let 
\begin{multline*}
A_j^{2,\e}=\{(\eta,\omega)\in I^\N\times\Omega\mid\,\exists\tau\in I^{j+r}
  \text{ such that }\tau\succ\eta|_j,\\ 
 \tau\ne\eta|_{j+r}\text{ and }0<k^{-rd}X_\tau(\omega)\le\e  X_{\eta|_j}(\omega)\}.
 \end{multline*}
Recall that, for any $\tau\in I^*$, we have
$P(X_\tau(\omega)>0\mid K_\omega\cap J_\tau\ne\emptyset)=1$ by 
\cite[Theorem 3.4]{MW}. Defining
\[
\begin{split}
 &A^s_{j,3}=\{(\eta,\omega)\in I^\N\times\Omega\mid 
   X_{\eta|_j}(\omega)>s\}\text{ and}\\
 &A_{j,4}^{s,\e}=\{(\eta,\omega)\in I^\N\times\Omega\mid\exists
   \tau\in  I^{j+r}\text{ such that }\tau\succ\eta|_j, \\
 &\phantom{kkkkkkkkkkkppppppppan}\tau\ne\eta|_{j+r}\text{ and }0<k^{-rd}
   X_\tau(\omega)\le\e s\},
 \end{split}
\]
we have
\[
H_j^{\alpha,\e,\delta}\le\1_{A_j^{1,\e}} + \1_{A_j^{2,\e}}
\le\1_{A_{j,1}^{s,\e}}+\1_{A_{j,2}^s}+\1_{A_{j,3}^s}+\1_{A_{j,4}^{s,\e}}.
\]
By Proposition~\ref{LLN}, the functions $\1_{A_{j,1}^{s,\e}}$ and $\1_{A_{j,3}^s}$
satisfy the strong law of large numbers. The same is true for $\1_{A_{j,2}^s}$ and
$\1_{A_{j,4}^{s,\e}}$ by Theorem~\ref{HRV}, since $A_{j,2}^s$ and $A_{i,2}^s$ as 
well as $A_{j,4}^{s,\e}$ and $A_{i,4}^{s,\e}$ are $Q$-independent if $|i-j|\ge r$. 
This can be seen similarly as in the proof of Lemma~\ref{independence}.  Hence,
we obtain the estimate
\begin{equation*}
\limsup_{n\to\infty}\frac 1n \sum_{j=1}^n H_j^{\alpha,\e,\delta}(\eta,\omega)\le
Q(A_{0,1}^{s,\e})+Q(A_{0,2}^s)+Q(A_{0,3}^s)+Q(A_{0,4}^{s,\e})
\end{equation*}
for $Q$-almost all $(\eta,\omega)\in I^{\N}\times\Omega$. Observe that the left
hand side of the above inequality decreases as $\e$ decreases for all 
$(\eta,\omega)\in I^{\N}\times\Omega$. For all large enough $s$, the value of 
$Q(A_{0,2}^s)+Q(A_{0,3}^s)$ is less than $\frac 12\delta$.
Fix such an $0<s<\infty$. According to \eqref{expectationXzero}, we have
$Q(X_r=0)=0$. Therefore, for all $\e$ small enough, we have
$Q(A_{0,1}^{s,\e})+Q(A_{0,4}^{s,\e})<\frac 12\delta$, completing the proof.
\end{proof}

Now we are ready to prove that the mean porosity of the natural measure equals
that of the Mandelbrot percolation set.

\begin{theorem}\label{equal}
For $Q$-almost all $(\eta,\omega)\in I^{\N}\times\Omega$, we have 
\[
\kappa(K_\omega,x(\eta),\alpha)=\kappa(\nu_\omega,x(\eta),\alpha)
\]
for all $\alpha\in (0,1)\setminus D$.
\end{theorem} 

\begin{proof}
For all $\alpha\in (0,1)$, $\e>0$ and $j\in\N$, let $\chi_j^\alpha$ and 
$\chi_j^{\alpha,\epsilon}$ be as in Definition~\ref{holeindicator} and 
$H_j^{\alpha,\e,\delta}$ as in Lemma~\ref{holedif}. The inequalities
\[
\underline{\kappa}(K_\omega,x(\eta),\alpha)\le
\underline{\kappa}(\nu_\omega,x(\eta),\alpha)\text{ and }
\overline{\kappa}(K_\omega,x(\eta),\alpha)\le
\overline{\kappa}(\nu_\omega,x(\eta),\alpha)
\]
are obvious for all $\alpha\in (0,1)$ and $(\eta,\omega)\in I^{\N}\times\Omega$
since $\chi_j^\alpha\le \chi_j^{\alpha,\epsilon}$. Therefore, for $Q$-almost all 
$(\eta,\omega)\in I^{\N}\times\Omega$, we have 
$\kappa(K_\omega,x(\eta),\alpha)\le\underline\kappa(\nu_\omega,x(\eta),\alpha)$
for all $\alpha\in (0,1)\setminus D$ by Theorem~\ref{meanporosityexists}. 

Let $0<\delta<\alpha$ and $\e>0$. Since 
$\chi_j^{\alpha,\e}\le\chi_j^{\alpha-\delta}+H_j^{\alpha,\e,\delta}$ for all $j\in\N$,
we have, by Lemma~\ref{holedif}, for $Q$-almost all 
$(\eta,\omega)\in I^{\N}\times\Omega$
\[
\begin{split}
\overline{\kappa}(\nu_\omega,x(\eta),\alpha)=&\lim_{\e\to 0}\limsup_{n\to\infty} 
  \frac 1n\sum_{j=1}^n\chi_j^{\alpha,\e}(\eta,\omega)\\
\le&\limsup_{n\to\infty}\frac 1n\sum_{j=1}^n\chi_j^{\alpha-\delta}(\eta,\omega)+
  \lim_{\e\to 0}\limsup_{n\to\infty}\frac 1n\sum_{j=1}^n
  H_j^{\alpha,\e,\delta}(\eta,\omega)\\
\le&\overline{\kappa}(K_\omega,x(\eta),\alpha-\delta)+\delta.
\end{split}
\]
Since $\alpha\mapsto\overline\kappa(K_\omega,x(\eta),\alpha)$ is continuous at
all $\alpha\in (0,1)\setminus D$ by Theorem~\ref{meanporosityexists}, we 
conclude that, for all $\alpha\in (0,1)\setminus D$, we have
for $Q$-almost all $(\eta,\omega)\in I^{\N}\times\Omega$ that
$\overline\kappa(\nu_\omega,x(\eta),\alpha)\le\kappa(K_\omega,x(\eta),\alpha)$.
As in the proof of Theorem~\ref{meanporosityexists}, we see that the order of 
the quantifiers may be reversed.
\end{proof}

Before stating a corollary of the previous theorem, we prove a lemma, which 
is well known, but for which we did not find a reference.

\begin{lemma}\label{zero}
Let $V$ be a coordinate hyperplane and let $e$ be the unit vector perpendicular
to $V$. Then, for all $t\in[0,1]$, 
\[
P\bigl(\nu_\omega(K_\omega\cap (te+V))>0\bigr)=0.
\]
\end{lemma}

\begin{proof}
Fix $t\in [0,1]$. According to \eqref{nudef},
\[
\nu_\omega(K_\omega\cap (te+V))=\lim_{j\to\infty}(\diam J_\emptyset)^d
\sum_{\substack{\tau\in I^j\\J_\tau\cap (te+V)\ne\emptyset}}k^{-jd}X_\tau(\omega)
  \1_{\{\omega(\tau)=c\}},
\]
and the above sequence decreases monotonically as $j$ tends to infinity. Hence, 
\begin{align*}
&E_P[\nu_\omega(K_\omega\cap (te+V))]\\
&\le(\diam J_\emptyset)^d
  \lim_{j\to\infty}E_P\Bigl[\sum_{\substack{\tau\in I^j\\J_\tau\cap (te+V)\ne\emptyset}}
  k^{-jd}X_\tau(\omega)\1_{\{\omega(\tau)=c\}}\Bigr].
\end{align*}
Note that, without the restriction $J_\tau\cap (te+V)\ne\emptyset$, the 
expectation on the right hand side equals 1. Since the restriction 
$J_\tau\cap (te+V)\ne\emptyset$ determines an exponentially decreasing 
proportion of indices as $j$ tends to infinity and since the random variables
$k^{-jd}X_\tau\1_{\{\omega(\tau)=c\}}$ have the same distribution, the limit of the
expectation equals 0.
\end{proof}

Next corollary is the counterpart of Corollary~\ref{positive} for mean 
porosities of the natural measure.

\begin{corollary}\label{mainmeasure}
For $P$-almost all $\omega\in\Omega$ and for $\nu_\omega$-almost all 
$x\in K_\omega$, we have
\[
0<\underline\kappa(\nu_\omega,x,\alpha)
\le \overline\kappa(\nu_\omega,x,\alpha)<1
\] 
for all $\alpha\in (0,1)$, $\kappa(\nu_\omega,x,0)=1$ and 
$\kappa(\nu_\omega,x,1)=0$. 
\end{corollary}

\begin{proof}
The first claim follows from Corollary~\ref{positive}, Theorem~\ref{equal} and
the monotonicity of the functions 
$\alpha\mapsto\underline\kappa(\nu_\omega,x,\alpha)$ and 
$\alpha\mapsto\overline\kappa(\nu_\omega,x,\alpha)$. Since 
$\underline\kappa(K_\omega,x,0)\le\underline\kappa(\nu_\omega,x,0)$, the second
claim follows from Corollary~\ref{positive}. Note that 
$\chi_j^{1,\epsilon}(\eta,\omega)=1$ only if $\nu_\omega(\partial J_{\eta|_j})>0$.
Therefore, the last claim follows from Lemma~\ref{zero}.  
\end{proof}

The following corollary solves completely Conjecture 3.2 stated in \cite{JJM}.
 
\begin{corollary}\label{infsup}
For $P$-almost all $\omega\in\Omega$ and for $\nu_\omega$-almost all 
$x\in K_\omega$, we have 
\[
\lpor(K_\omega,x)=\lpor(\nu_\omega,x)=0,\quad\upor(K_\omega,x)=\frac 12
\text{ and }\upor(\nu_\omega,x)=1.
\]
\end{corollary}

\begin{proof} 
By Corollary~\ref{positive}, for $Q$-almost all 
$(\eta,\omega)\in I^{\N}\times\Omega$, we have that 
$\overline\kappa(K_\omega,x(\eta),\alpha)<1$ for all $\alpha>0$. Hence, for 
$P$-almost all $\omega\in\Omega$ and for $\nu_\omega$-almost all $x\in K_\omega$, 
there are, for all $\alpha>0$, arbitrarily large $i\in\N$ such that $K_\omega$
does not have an $\alpha$-hole at scale $i$ near $x$ which is contained in
$Q_i^k(x)$. Note that, in Definitions~\ref{defporoset} and \ref{defporomeas}, 
the holes are defined using balls while the mean porosities are defined in 
terms of $k$-adic cubes (see Definitions~\ref{ourmean} and \ref{defmean}).  
Therefore, it is possible that $B(x,k^{-i})$ contains an $\alpha$-hole which is
outside the construction cube $Q_i^k(x)$ if $x$ is 
close to the boundary of $Q_i^k(x)$. We show that there are infinitely many 
$i\in\N$ such that this will not happen. 

Fix $\alpha\in(0,\frac 14)$ and $r>8$ large enough so that $2k^{-r}<\alpha$. 
Let $I'\subset I^r$ be the set of words such that, for all $\tau\in I'$, the 
$\varrho$-distance from all points of $J_\tau$ to the centre of $J_\emptyset$ is 
at most $\frac 14$. For all $i\in\N$, define 
$Y_i^\alpha\colon I^{\N}\times\Omega\to\{0,1\}$ by setting 
$Y_i^\alpha(\eta,\omega)=1$, if and only if $J_{\eta|_i}$ is chosen, $\eta|_{i+r}$ 
ends with a word from $I'$ and $K_\omega$ does not have an $\frac 12\alpha$-hole
at scale $i$ near $x(\eta)$ which is completely inside 
$J_{\eta|_i}\setminus J_{\eta|_{i+r}}$. Note that if $x(\eta)\in K_\omega$ and 
$K_\omega$ has an $\alpha$-hole 
at scale $i$ near $x(\eta)$, then at least half of this hole is in 
$J_{\eta|_i}\setminus J_{\eta|_{i+r}}$. Thus, $K_\omega$ does not have an 
$\alpha$-hole at scale $i$ near $x(\eta)$ if $Y_i^\alpha(\eta,\omega)=1$.
Since for indices $i$ and $j$ with $|i-j|\ge r$, the events 
$\{(\eta,\omega)\in I^{\N}\times\Omega\mid Y_i^\alpha(\eta,\omega)=1\}$ and
$\{(\eta,\omega)\in I^{\N}\times\Omega\mid Y_j^\alpha(\eta,\omega)=1\}$
are $Q$-independent (compare with Lemma~\ref{independence}),
the averages of random variables $Y_i^\alpha(\eta,\omega)$ converge 
to $E_Q(Y_0^\alpha)>0$ for $Q$-almost all $(\eta,\omega)\in I^\N\times\Omega$. If 
$Y_i^\alpha(\eta,\omega)=1$, then $B(x(\eta),\frac 14k^{-i})\subset J_{\eta|_i}$ 
and there is no $z\in B(x(\eta),\frac 14k^{-i})$ such that 
$B(z,\frac 12\alpha k^{-i})\subset B(x(\eta),\frac 14k^{-i})\setminus K_\omega$.
Therefore, $\por(K_\omega,x(\eta),\frac 14k^{-i})\le 2\alpha$. A similar argument 
shows that $\por(\nu_\omega,x(\eta),\frac 14k^{-i})\le 2\alpha$. Let 
$(\alpha_j)_{j\in\N}$ and $(\e_k)_{k\in\N}$ be sequences tending to 0.
For $Q$-almost all $(\eta,\omega)\in I^\N\times\Omega$, we have for all 
$j,k\in\N$ that there are infinitely many scales $i\in\N$ such that
\[
\por(K_\omega,x(\eta),\frac 14k^{-i})<\alpha_j\text{ and } 
 \por(\nu_\omega,x(\eta),\frac 14k^{-i},\varepsilon_k)<\alpha_j.
\]
Thus, we conclude that
\[
\lpor(K_\omega,x)=0=\lpor(\nu_\omega,x)
\]
for $P$-almost all $\omega\in\Omega$ and for $\nu_\omega$-almost all 
$x\in K_\omega$. 

Since 
$\underline{\kappa}(K_\omega,x,\alpha)>0$ and 
$\underline{\kappa}(\nu_\omega,x,\alpha)>0$ for all $\alpha<1$, we deduce that 
\[
\upor(K_\omega,x)=\frac 12\text{ and }\upor(\nu_\omega,x)=1
\]
for $P$-almost all $\omega\in\Omega$ and for $\nu_\omega$-almost all 
$x\in K_\omega$.
\end{proof}

\begin{remark}
Theorems~\ref{meanporosityexists} and \ref{equal} should extend to 
homogeneous random self-similar sets satisfying the random strong open set 
condition.
\end{remark}

\providecommand{\bysame}{\leavevmode\hbox
  to3em{\hrulefill}\thinspace}
\providecommand{\MR}{\relax\ifhmode\unskip\space\fi MR }
\providecommand{\MRhref}[2]{%
  \href{http://www.ams.org/mathscinet-getitem?mr=#1}{#2} }
\providecommand{\href}[2]{#2}

\end{document}